\documentclass{conm-p-l}
\usepackage{amsmath,amssymb,amsfonts,verbatim}

\newtheorem{teo}{Theorem}[section]
\newtheorem{lem}[teo]{Lemma}
\theoremstyle{remark}
\newtheorem{rem}[teo]{Remark}
\newtheorem{prop}[teo]{Proposition }
\newtheorem{notation}[teo]{Notation }
\newtheorem{cor}[teo]{Corollary }
\theoremstyle{definition}
\newtheorem{defi}[teo]{Definition }
\newtheorem{ex}[teo]{Example }

\numberwithin{equation}{section}

\DeclareMathOperator{\rk}{rk}
\DeclareMathOperator{\Ker}{Ker}

\newcommand{\C}{{\mathbb C}}
\newcommand{\Q}{{\mathbb Q}}
\newcommand{\N}{{\mathbb N}}

\newcommand{\K}{{\mathbb K}}

\newcommand{\calI}{\mathcal I} \newcommand{\calT}{\mathcal T}
\newcommand{\calM}{\mathcal M} \newcommand{\calN}{\mathcal N}
\newcommand{\calB}{\mathcal B} \newcommand{\calP}{\mathcal P}
\newcommand{\calG}{\mathcal G} \newcommand{\calR}{\mathcal R}
 \newcommand{\calV}{\mathcal V}
\newcommand{\calW}{\mathcal W}
 \newcommand{\pp}{\ensuremath {\mathbb{P}}}

\begin{document}

\title{Ideals of curves given by points}\thanks{This research was partially supported by 
M.I.U.R.  and by G.N.S.A.G.A.}

\author{E. Fortuna}
\address{Dipartimento di Matematica, Universit\`a
  di Pisa, Largo B. Pontecorvo 5, I-56127 Pisa,
  Italy}
\email{fortuna@dm.unipi.it}

\author{P. Gianni}
\address{Dipartimento di Matematica, Universit\`a
  di Pisa, Largo B. Pontecorvo 5, I-56127 Pisa,
  Italy}
\email{gianni@dm.unipi.it}

\author{B. Trager}
\address{IBM T.J.Watson Research Center, 1101 Kitchawan Road,
  Yorktown Heights, NY 10598, USA}
\email{bmt@us.ibm.com}



\begin{abstract}   Let $C$ be an irreducible projective curve of degree $d$ in 
$\pp^n(\K)$, where $\K$  is an algebraically closed
field, and let $I$ be the associated homogeneous prime ideal. We wish to compute generators for $I$, assuming we are given sufficiently many points on the curve $C$. In particular if $I$ can be generated by polynomials of degree at most $m$ and we are
given $m d + 1$ points on $C$, then we can find a set of generators
for $I$.  We will show that a minimal set of generators of $I$ can be constructed in polynomial time. Our constructions are completely independent of any notion of term 
ordering; this allows us the maximal freedom in performing our constructions in order to 
improve the numerical stability.  We also summarize some classical results on bounds for the degrees of the generators of our ideal in terms of the degree and genus of the curve.
\end{abstract}

\maketitle

\section{Introduction}
Let $C$ be an irreducible projective curve of degree $d$ in
$\pp^n(\K)$, where $\K$  is an algebraically closed
field, and let $I=\calI(C)$ be the associated homogeneous prime ideal
of $\mathcal P=\K[x_0, \ldots, x_n]$ consisting of all the polynomials
vanishing on $C$.  We wish to compute generators for $I$, assuming we
are given sufficiently many points on the curve $C$. In particular if
$I$ can be generated by polynomials of degree at most $m$ and we are
given at least $m d + 1$ points on $C$, then we can find a set of generators
for $I$. It is a simple consequence of Bezout's theorem that any
polynomial of degree $k$ which vanishes on more than $k d $ points of
$C$ must be contained in $I$.  Although the number of monomials in $n+1$ variables
of degree at most $m$ is not polynomial in both $n$ and $m$, we will present a process
which constructs generators degree by degree and results in a
polynomial time algorithm for computing generators for $I$. Polynomial time algorithms
for computing Gr\"obner bases of ideals of affine points were presented in \cite{Moller}, and
then extended to minimal generators of ideals of projective points in \cite{Marinari}.
These algorithms require exact arithmetic, and assume a term ordering is given.
We present new algorithms which are completely independent of any notion of term
ordering. We believe that this flexibility is necessary when working with approximate
coefficients.

Given a homogeneous ideal $I$, there are many different choices for monomials representing
cosets of $\calP/I$, i.e. for a complement of $I$. It is well known that the natural coset representatives associated
with Gr\"{o}bner bases do not remain stable with respect to small coefficient
perturbations of the ideal generators. Border bases were introduced to help
overcome this problem (\cite{Robbiano}). Given a fixed choice of complement for a zero-dimensional
ideal, its border basis is uniquely determined in contrast with Gr\"obner bases
where the complement is uniquely determined by the given term ordering.
Border bases are usually defined for zero-dimensional ideals, which guarantees a finite
basis. We extend the definition to homogeneous ideals of any dimension, but bound the degree
in order to preserve finiteness. In much of the literature
on border bases, the complements are required to be closed under division by variables.
This makes complements of border bases more similar to complements of Gr\"obner bases
which also have this property.
In particular this is done in \cite{Heldt}, therefore their
algorithms need to explicitly decide whether or not candidate leading
monomials have coefficients which are so small that they should be treated as zero.
As suggested by Mourrain and Tr\'ebuchet (\cite{Mourrain}), we only require the complement
to be connected to 1. This means we require that each complement monomial of degree $i$
is a multiple of some complement monomial of degree $i-1$. This extra flexibility
in the choice of complement monomials means that we can use standard numerical software like
the QR algorithm with column pivoting (QRP) (\cite{Golub}) to choose our complement in each degree. 
One could define a complement to be any set of representatives for $\calP/I$, but with this
definition we would not be able to obtain algorithms which are polynomial in both the
degree of the curve and the number of variables. In particular, requiring the complement
to be connected to 1 implies a strong condition on the syzygy module. We show 
that the syzygy module for a vector space basis of a homogeneous ideal whose
complement is connected to 1 is generated by vectors  whose entries
have degree at most one, generalizing the result of Mourrain and
Tr\'ebuchet for border bases of zero-dimensional ideals. These special
generators of the syzygy 
module can be used to obtain  a polynomial time algorithm for
constructing minimal generators for our ideals. 
The algorithms developed by Cioffi \cite{Cioffi} have a similar complexity in the case of
exact coefficients but her use of Gr\"{o}bner bases requires a
term ordering which determines a unique complement 
which may be numerically unstable when working with approximate points.

The motivation for this paper came from the desire to be able
to compute the generators of space curves starting from numerical
software which generates points on curves, in particular 
computing points on canonical or bicanonical models
for Riemann surfaces presented as Fuchsian groups (\cite{Seppala}).
Our intended applications differ from that of \cite{Heldt} and
\cite{Abbott} since we assume that the points defining our curve are
generated by a numerical algorithm whose errors tend to be very small
as opposed to empirical measurements whose errors could be much
larger. Thus our goal is to present algorithms based on numerically
stable constructions like SVD for computing ideal generators and QRP
for deciding which monomials represent the complement of our ideal.

In section 2 of this paper we derive some general properties of
border bases and complements for general homogeneous ideals. Using
these properties we present a polynomial time algorithm for finding
a minimal basis for a homogeneous ideal. In section 3 we specialize to
ideals of curves given by points, and use point evaluation matrices to complete the task
of computing border bases for homogeneous ideals of curves. 
Assuming exact arithmetic of unit cost, 
we also provide an overall complexity analysis of the border basis algorithm
and the minimal basis algorithm. In section 4 we consider the situation of approximate
points and show how our algorithms can be adapted to use standard numerical software,
where we allow numerical algorithms like the QRP to choose the complement monomials
in order to help improve the numerical stability. In general the stability
also strongly depends on how the points are distributed, and we feel this is
an interesting problem for future research, along with the possibility of using other
approaches such as interval arithmetic.

In order to use Bezout's theorem, we assume we have at least a bound
on the degree of our curve. In the last section we summarize some
classical results on bounds for the degrees of generators of our ideal
in terms of the degree and genus of the curve. In particular, as shown
by Petri (\cite{Petri}), a non-hyperelliptic canonical curve of genus $g \ge 4$ can
be generated in degrees 2 and 3. We also recall the completely general
result of Gruson-Lazarsfeld-Peskine (\cite{Peskine}) which shows that any
non-degenerate curve of degree $d$ in $\pp^n(\K)$ can be generated in
degree $d-n+2$.

Sometimes we have additional information about the nature of the curve
whose points we are given.  For instance, if we know the Hilbert
function, the rank of our point evaluation matrices is explicitly given
instead of being determined by examining its singular value
spectrum.

\section{Border bases for homogeneous ideals}\label{BB}

Let $\K$ be an algebraically closed field.  For any $s\in \N$, let
$\calT_s$ be the set of all terms of degree $s$ in $\mathcal P=\K[x_0,
\ldots, x_n]$ and let $\mathcal P_s$ be the vector subspace of
$\mathcal P$ generated by $\calT_s$. Recall that $\dim \mathcal
P_s=\binom {n+s}s$.  We will also set $\mathcal P_{\leq s}=
\oplus_{i=0}^s \mathcal P_i$.

\begin{notation} We will use the following notation:
\begin{enumerate}
\item For any subset $Y\subset \pp^n(\K)$ denote by $\calI(Y)$ the
radical homogeneous ideal of $\calP$ consisting of all the polynomials
vanishing on $Y$.
\item For any homogeneous ideal $I$ in $\mathcal P$, let $I_s=I \cap
  \mathcal P_s$ (so that $I=\oplus_{s\geq 0} I_s$) and let $I_{\leq
  s}=I \cap \mathcal P_{\leq s}$.
\item For any $S\subset\mathcal P$, denote by $I(S)$ the ideal
  generated by $S$.
\item For any  $S\subset\mathcal P_s$, denote by $ \langle S \rangle $ the vector
  subspace of $\mathcal P_s$ generated by $S$.
\item For any  $S\subset\mathcal P_s$, let $S^+=\bigcup_{j=0}^n
  (x_j\,S)\subset \mathcal P_{s+1}$.
\item If ${\bf a}=(a_1, \ldots, a_h)\in \K^h$ and $\mathcal F =[F_1,
  \ldots, F_h]$ is a list of polynomials, we set $${\bf a} \cdot
  \mathcal F= a_1F_1 + \ldots+ a_h F_h.$$
\item For any finite set $A$, we denote by $|A|$ its cardinality.
\end{enumerate}
\end{notation}

\begin{defi}\label{complement} Let $J$ be a proper homogeneous ideal in
  $\mathcal P$ and $s\in \N$.  Let $\calN_0 = \{1\}$ and, for each
  $k=1, \ldots, s$, assume that $\calN_k$ is a set of monomials in
  $\calT_k$ such that
$$\calN_k \subset \calN_{k-1}^+ \quad \mbox{and} \quad \calP_k = J_k
\;\oplus  \langle \calN_k \rangle .$$ We call $\calN=\{\calN_0, \ldots, \calN_s\}$ 
a \emph{complement} of the ideal $J$ up to degree $s$.
\end{defi}

\begin{rem}\label{compl-prop} Let $\calN=\{\calN_0, \ldots,
    \calN_s\}$ be a complement of a proper homogeneous ideal $J$ up to
degree $s$. Then: 
\begin{enumerate}
\item the condition $\calN_k \subset \calN_{k-1}^+$ implies that 
$\calN$ is connected to $1$, i.e. for each
  $m\in \calN$ there exist variables $x_{i_1}, \ldots, x_{i_k}$ such
  that $m=x_{i_1}\cdot \ldots\cdot x_{i_k}$ and $x_{i_1}\cdot
  \ldots\cdot x_{i_j}\in \calN$ for each $j < k$,
\item from the
definition it follows that, for $k=1, \ldots, s$,
$$ \langle \calN_{k-1}^+ \rangle =(J_k \;\cap  \langle \calN_{k-1}^+ \rangle ) \;\oplus  \langle \calN_k \rangle .$$
\end{enumerate}
\end{rem}\qed

\begin{lem}\label{iteration-compl} Let $J$ be a proper homogeneous ideal of
  $\calP$. Assume that $\calN_{k-1} \subseteq \calT_{k-1}$ and $\calN_k
  \subseteq \calT_k$ are sets of monomials such that
\begin{enumerate}
\item[(a)] $\calP_{k-1}= J_{k-1} \,\oplus  \langle \calN_{k-1} \rangle $
\item[(b)] $ \langle \calN_{k-1}^+ \rangle =(J_k \,\cap  \langle \calN_{k-1}^+ \rangle ) \,\oplus  \langle \calN_k \rangle $.
\end{enumerate}
Then
\begin{enumerate}
\item $J_k=  \langle J_{k-1}^+ \rangle  + \,(J_k \, \cap  \langle \calN_{k-1}^+ \rangle )$
\item $\calP_k = J_k\,\oplus  \langle \calN_k \rangle $.
\end{enumerate}
\end{lem}
\begin{proof} By  hypothesis (a), we have
$$\calP_k =  \langle \calP_{k-1} ^+ \rangle  =   \langle J_{k-1}^+ \rangle  +  \langle \calN_{k-1}^+ \rangle $$
and, since $J_k \supseteq  \langle J_{k-1}^+ \rangle $, we have also
$$
J_k=\calP_k \cap J_k=  \langle J_{k-1}^+ \rangle  + (J_k \, \cap  \langle \calN_{k-1}^+ \rangle ),
$$
which proves (1).
Hence, by the previous relations and hypothesis (b), we have 
$$\calP_k =  \langle J_{k-1}^+\rangle + \langle \calN_{k-1}^+\rangle = 
\langle J_{k-1}^+\rangle + (J_k \,\cap \langle \calN_{k-1}^+ \rangle ) + \langle \calN_k \rangle =
J_k +  \langle \calN_k \rangle .$$
On the other hand, again by hypothesis (b) we have that $ \langle \calN_k \rangle 
\subseteq  \langle \calN_{k-1}^+ \rangle $ and hence 
$$J_k\, \cap  \langle \calN_k \rangle  \subseteq J_k \, \cap  \langle \calN_{k-1}^+ \rangle  \cap
 \langle \calN_k \rangle  =\{0\},$$
which completes the proof of (2).
\end{proof}

\begin{rem}\label{complsuffices} It 
is always possible to choose a complement up to any fixed degree for
any proper homogeneous ideal $J$ of $\calP$ incrementally. Namely, if
$\calN =\{\calN_0, \ldots, \calN_{k-1}\}$ is a complement of $J$ up to degree
$k-1$, it is sufficient to choose a set $\calN_k \subseteq \calT_k$
such that $ \langle \calN_{k-1}^+ \rangle =(J_k \,\cap  \langle \calN_{k-1}^+ \rangle ) \,\oplus
 \langle \calN_k \rangle $: then Lemma \ref{iteration-compl} assures that $\calP_k =
J_k\,\oplus  \langle \calN_k \rangle $ and hence that $\calN=\{\calN_0, \ldots,
\calN_k\}$ is a complement of the ideal $J$ up to degree $k$.

Moreover, again by Lemma \ref{iteration-compl}, once one has a
complement $\calN$ of $J$ up to any fixed degree $s$, one can get a
set of generators of the ideal $I(J_1, \ldots, J_s)$ as the union of
sets of generators of $J_k \,\cap  \langle \calN_{k-1}^+ \rangle $ for $k=1, \ldots,
s$.  


\end{rem}

\begin{defi}\label{border-basis} Let $J$ be a proper homogeneous ideal in
  $\mathcal P$ and assume that $\calN=\{\calN_0, \ldots, \calN_s\}$ is
a complement of  $J$ up to degree $s$.
\begin{enumerate}
\item For all $k=1, \ldots, s$ let $(\partial \calN)_k= \calN_{k-1}^+
\setminus \calN_k$; the elements in $(\partial \calN)_k$ will be
called \emph{border monomials} in degree $k$.
\item For each $m\in (\partial \calN)_k$ let $\psi(m)$ be the unique
  polynomial in $ \langle \calN_k \rangle $ such that $m + \psi(m)\in J_k$; we will
  call the homogeneous polynomial $m + \psi(m)$ the \emph{border polynomial
  associated to $m$}.
\item If $\calB_k$ denotes the set of all border polynomials of degree
  $k$, the set $\calB=\calB_1 \cup \ldots \cup \calB_s$ is called the
  \emph{border basis of $J$ up to degree $s$ associated to $\calN$}.
\end{enumerate}
\end{defi}

\begin{rem}\label{BBasis}  This notion of bounded degree
  border basis applies to arbitrary homogeneous ideals and coincides
  with the classical notion of border basis (\cite{Robbiano}) in the case of
  homogeneous zero-dimensional ideals, provided we choose $s$ to be
  larger than the maximal degree of any monomial in the (finite)
  complement $\calN$ of $J$.
\qed\end{rem}

\begin{prop}\label{Bkgenerates} Assume that $\calN$ is
a complement of a proper homogeneous ideal  $J$ up to degree $s$ and let 
$\calB=\calB_1 \cup \ldots \cup \calB_s$ be the associated border
basis. Then 
\begin{enumerate}
\item $\calB_k$ is a basis of the vector space
$J_k \,\cap
 \langle \calN_{k-1}^+ \rangle $ for each  $k=1, \ldots, s$,
\item $\calB$ is a set of generators of the ideal $I(J_1, \ldots,
  J_s)$.
\end{enumerate}
\end{prop}
\begin{proof} (1) Let $b\in J_k \,\cap  \langle \calN_{k-1}^+ \rangle $. Since
  $\calN_{k-1}^+ = (\partial \calN)_k \cup \calN_k$, there exist
 $a_i, b_i \in \K$ such that 
$$b=\sum_{m_i\in (\partial \calN)_k} a_i m_i  + 
\sum_{m_i\in \calN_k} b_i m_i.$$
For each $m_i\in (\partial \calN)_k$ let $\psi(m_i)$ be the unique
  polynomial in $ \langle \calN_k \rangle $ such that $m_i + \psi(m_i)\in J_k$. Then
  we can write $b=u+v$ with
$$ u=\sum_{m_i\in (\partial \calN)_k} a_i (m_i + \psi(m_i)), \qquad
v=\sum_{m_i\in \calN_k}b_im_i - \sum_{m_i\in (\partial
\calN)_k}a_i\psi(m_i).$$ Note that $u\in J_k \,\cap  \langle \calB_k \rangle $ and
$v\in  \langle \calN_k \rangle $. Since $b\in J_k$, then we have that $v=b-u\in J_k
\,\cap  \langle \calN_k \rangle  =\{0\}$. Thus $b=u$ and hence $b\in  \langle \calB_k \rangle $.

(2) follows immediately from (1) and Remark \ref{complsuffices}. 
\end{proof}

We now suggest a simple method to construct 
recursively both a complement and a border basis of a homogeneous
ideal $J$ up to any fixed degree. Even if the ideal may not have an explicit representation,  
if we assume the  capability of computing a basis of its intersection with a vector subspace
generated by a finite set $N$ of monomials of the same degree, we will be able to compute a border basis for $J$ up to any fixed degree.
 We will denote this condition by saying that the ideal is {\it represented } by  the function  
$ComputeBasis_J$, which, for any such $N$, returns a basis for the intersection $J \cap \langle N \rangle$.  In the next section we will see that such a function can be easily computed for ideals of points.

In the description of the algorithms we will use the following notations:
\begin{enumerate}
\item[a.] If $v $ is a polynomial and $S=\{n_1,\ldots,n_t\}$ is a set of monomials, 
then coeffs$(v,S)$ will denote the vector $(a_1,\ldots,a_t)$ of the coefficients of the
monomials of $S$ in $v$.

\item[b.]  Given a matrix $A$, we will denote $RRE(A) =(E,\Sigma)$ where 
\begin{enumerate}
\item[-] $E$ is the completely reduced row echelon form of $A$ (i.e. each pivot is equal to 1, and in each of the columns containing a pivot all the elements different from the pivot are zero) 
\item[-] $\Sigma$ is the
set consisting of the indexes of the columns containing the pivots of $E$ .
\end{enumerate}
\end{enumerate}

\underline{Algorithm BorderBasisWithComplement}

\underline {Input:} 
\begin{enumerate} 
\item[-] a function  $ComputeBasis_J$ representing a  homogeneous ideal $J$ 
\item[-] $s \in \N$
\end{enumerate}

\underline{Output:}
\begin{enumerate}
\item [-] $\{\calN_0, \ldots, \calN_{s}\}$ a complement of $J$ up to degree $s$
\item[-] $\{\calB_1, \ldots , \calB_s\}$ the associated border basis.
\end{enumerate}

\underline{Procedure:}
\begin{enumerate}
\item[-] $\calN_0 = \{1\}$ 
\item[-] for $k=1..s$ repeat
\begin{enumerate}
\item[]--- construct the set of distinct monomials $\calN_{k-1}^+ = \{ m_1 ,\ldots, m_t\}$
\item[] $\calN_{k-1}^+ := \{x_i  m \ |\ 1\leq i \leq n , m \in \calN_{k-1}\}$
\item[] $\calV _k:= ComputeBasis_J(\calN_{k-1}^+)$
\item[] $q:=| \calV_k |$
\item[] $t:= | \calN_{k-1}^+ |$
\item[] --- compute the $q \times t$  matrix with rows coeffs$(v,\calN_{k-1}^+)$ for $v \in \calV_k$
\item[] $A:= matrix ($coeffs$(v, \calN_{k-1}^+) \ |\ v \in \calV _k)$ 
\item[] $(E,\Sigma) := RRE(A)$
\item[] --- the monomial with index not in $\Sigma$ are put in $\calN_k$ 
\item[] $\calN_k:= \{m_j \ | \ j \notin \Sigma\}$
\item[] --- every row represents an element of $\calB_k$
\item[] $\calB_k:= \{\sum_j e_{i,j} m_j\ |  i=1, \ldots, q\}$
\end{enumerate}
\end{enumerate}

\begin{prop}\label{complinP} Given a proper homogeneous ideal $J$ in
  $\mathcal P$ represented by a function  {\it $ComputeBasis_J$}  and $s \in \N$,
 the algorithm BorderBasisWithComplement constructs a complement and a border
  basis for $J$ up to degree $s$.
\end{prop}
\begin{proof} At each step,  the rows of $E$  correspond to polynomials of 
the form $ m+\psi (m)$, with $m \not\in \calN_k$ and $\psi (m) \in \calN_k$, which are  a basis of
$J_{k} \,\cap  \langle \calN_{k-1}^+ \rangle $. Moreover , since the monomials in  
$\calN_k$ correspond to the non-pivot positions, by construction we have that $  \langle \calN_{k-1}^+ \rangle =(J_{k} \,\cap  \langle \calN_{k-1}^+ \rangle ) \, \oplus  \langle \calN_k \rangle $.

 Then by Lemma \ref{iteration-compl} we get that 
$\calP_k = J_k \, \oplus  \langle \calN_k \rangle $; hence $\{\calN_0, \ldots,
\calN_k\}$ is a complement of $J$ up to degree $k$ and $\calB_1 \cup \ldots \cup \calB_k$ is the associated border basis up to degree $k$.
 \end{proof}

By Proposition \ref{Bkgenerates} a
border basis $\calB$ of $J$ up to degree $s$ is a set of generators of the
ideal $I(J_1, \ldots, J_s)$, but in general it is not minimal. We will
see how one can eliminate redundant polynomials in $\calB$ so as to
obtain a minimal set of generators of that ideal.

\begin{prop}\label{linear-gen} Let $J$ be a proper homogeneous ideal in
  $\mathcal P$. Assume that $\calN=\{\calN_0, \ldots, \calN_s\}$ is
a complement of  $J$ up to degree $s$ and let 
$\calB=\calB_1 \cup \ldots \cup \calB_s$ be the associated border
basis. Then:
\begin{enumerate}
\item The ideal $L=I(\calB, \calN_s^+)$ is homogeneous and
  zero-dimensional, and $\calB \, \cup \calN_s^+$ is the border basis of
  $L$  associated to its complement $\calN$.
\item The module $Syz (\calB, \calN_s^+)$ is generated by vectors
  whose entries have degree at most $1$.
\end{enumerate} 
\end{prop}
\begin{proof} (1)  Since $L_s=J_s$, we have that 
$\calP_s = L_s \, \oplus  \langle \calN_s \rangle $. Moreover, since $ \calN_s^+
  \subseteq L$, choosing $\calN_{s+1}=\emptyset$ we see that the
  triple $L, \calN_s, \calN_{s+1}$ satisfies the hypotheses of Lemma
  \ref{iteration-compl} (we set $ \langle \emptyset  \rangle = \{0\}$). Thus we get
  that $\calP_{s+1}= L_{s+1}$; so the ideal $L$
  is zero-dimensional. If we set $\calN_j =\emptyset$ for all $j\geq
  s+1$, then $\widetilde \calN= \{\calN_j\}_{j\in\N}$ is a complement
  of $L$. Moreover, since $(\partial
  \calN)_{s+1}=\calN_s^+$, if we let $\calB_{s+1}=\calN_s^+$, then
  $\calB \cup \calN_s^+$ is a border basis of $L$.

  (2) Since $\calN_{k+1}\subseteq \calN_k^+$ for each $k$ and $\calN_0
  = \{1\}$, we have that $\calN$ is connected to $1$
  (see Remark \ref{compl-prop}).  Thus it is possible to
  apply Theorem 4.3 in \cite{Mourrain}, which implies that the syzygies
  among the elements of a border basis of a zero-dimensional ideal
  with a complement connected to $1$ can be generated by syzygies
  whose coefficients have degree at most $1$. 
\end{proof}

By changing bases in each constant degree subspace we obtain the following more general result:

\begin{prop}\label{linear-gen-2} Let $J$ be a proper homogeneous ideal in
  $\mathcal P$. Assume that $\calN=\{\calN_0, \ldots, \calN_s\}$ is
a complement of  $J$ up to degree $s$. Let $\calV=\calV_1 \cup \ldots \cup \calV_s$ where $\calV_k$ is a basis of
$J_k \,\cap  \langle \calN ^+_{k-1} \rangle $ for each $k=1, \ldots, s$. Then 
the module $Syz (\calV, \calN_s^+)$ is generated by vectors
  whose entries have degree at most $1$.
\end{prop}

The following corollary shows that the redundant elements in $\calV_s$
can be expressed as a combination of elements in $\calV_{s-1}$ and the
other elements in  $\calV_s$:

\begin{cor}\label{gen-syz} Under the hypotheses of Proposition
    \ref{linear-gen-2}, let $f\in \calV_s$ and denote 
    $\calW_s=\calV_s\setminus \{f\}$. If   
$f\in I(\calV_1, \ldots , \calV_{s-1}, \calW_s)$, then 
$f\in  \langle \calV_{s-1} ^+, \calW_s \rangle $.
\end{cor}
\begin{proof} By hypothesis there exists a syzygy among the elements
  of $\calV= \calV_1 \cup \ldots \cup \calV_s$ such that the
  coefficient of $f$ is a non-zero constant. Hence by 
  Proposition \ref{linear-gen-2}, there exists a homogeneous generator
  of $Syz (\calV, \calN_s^+)$ whose entries have degree at most $1$
  and where the coefficient of $f$ is a non-zero constant. Since $\deg
  f=k$ and the coefficient of $f$ is constant, in this generating
  syzygy only the elements of $ \calV_{k-1} \cup \calV_k$ can have
  non-zero coefficients.
\end{proof}

We now describe two methods  to construct a
minimal set of  generators of the ideal $I(J_1, \ldots, J_s)$  depending on 
whether we start with a set of generators which form a border basis or not.

\begin{prop}\label{minimize-1}  Let $s \in \N$ and let  $\calN=\{\calN_0, \ldots, \calN_s\}$  be a complement  up to degree $s$ of a  proper homogeneous ideal  $J$ in
$\mathcal P$. Given $ \{\calV_1, \ldots, \calV_{s}\}$ where $\calV_k$ is a basis of  $J \cap \langle \calN_{k-1}^+ \rangle $, then for each $k=1, \ldots, s$ it is possible to construct a set of polynomials $\calG_k \subseteq \calV_ k$  such
  that  $\calG_1 \cup \ldots \cup \calG_s$  is a minimal set of
  generators of the ideal $I(J_1, \ldots, J_s)$.
\end{prop}
\begin{proof} Let $\calG_1= \calV_1$ and assume that  $\calG_1 \cup
  \ldots \cup\calG_{k-1}$ is a minimal set of generators of $I(J_1,
\ldots, J_{k-1})$ with $\calG_i \subseteq \calV_i$ for $i=1, \ldots,
k-1$.

Note that a polynomial $f\in \calV_k$ is redundant w.r.t.  $\calG_1
\cup \ldots \cup \calG_{k-1} \cup \calV_k$ if and only if it is
redundant w.r.t.  $\calV_1 \cup \ldots \cup \calV_k$.  Thus, by
Corollary \ref{gen-syz} it suffices to look for linear relations among
the elements of $\calV_{k-1}^+ \cup
\calV_k$ and for a set $\calG_k \subseteq \calV_k$ such that 
$\langle \calV_{k-1}^+\rangle + \langle \calV_k \rangle = 
\langle \calV_{k-1}^+\rangle \oplus \langle \calG_k \rangle$. Hence it is  sufficient  to find a basis of $\langle \calV_{k-1}^+ \rangle \cap \langle \calV_k \rangle$,
 extend it with elements $w_1,\ldots,w_t$ to a basis of $\langle \calV_k \rangle$ and define $\calG_k=\{w_1,\ldots,w_t\}$.

In order to compute the intersection $\langle \calV_{k-1}^+\rangle \cap \langle \calV_k \rangle$
consider the monomial basis $S = S_1 \cup S_2 \cup S_3$ of $\langle \calV_{k-1}^+\rangle + \langle \calV_k \rangle$, where $S_1=(\partial \calN)_{k-1}^+ \setminus \calN_{k-1}^+$,   $S_2=(\partial
\calN)_k$ and $S_3=\calN_k$. 
Observe that $\calN_{k-1}^+ = (\partial \calN)_k \cup \calN_k$ and
$\calV_{k-1}^+  \subset \langle (\partial \calN)_{k-1}^+ \rangle + \langle \calN_{k-1}^+ \rangle$.
Let $s_i=|S_i|$, for $i=1,2,3$, and
$l=|\calV_{k-1}^+|$; with this notation $|\calV_k|=|(\partial
\calN)_k|=s_2$.

Let  $U$ be the matrix
whose columns contain the coefficients of the polynomials of
$  \calV_{k-1}^+ \cup  \calV_k $ with respect to $S$.  Thus $U$ is a 
$(s_1+s_2+s_3) \times (l+s_2)$
block matrix of the form
$$U=
\left(\begin{array}{c|c}
U_1 & 0\\ \hline
U_2 & U_3\\ \hline
U_4 & U_5
\end{array}\right)$$
\noindent
and, if we denote by $\pi_2:\K^l \times \K^{s_2} \rightarrow \K^{s_2}$ the
projection on the last  $s_2$ coordinates,  the vectors of $\pi_2(\Ker  U)$ are
the coordinates (w.r.t. $\calV_k$) of the vectors of 
$\langle \calV_{k-1}^+\rangle \cap \langle \calV_k \rangle$.

In order to compute $\Ker U$ we can reduce ourselves to consider the matrix 
$$ \widetilde U= \left(\begin{array}{c|c}
U_1 & 0\\ \hline
U_2 & U_3
\end{array}\right). $$
Namely, since $S_3 = \calN_k$, if 
$\widetilde U v=0$, then 
$U v \in J_k \,\cap  \langle \calN_k \rangle =\{0\}$,  hence $\Ker U = \Ker \widetilde U$.

In order to finish the construction it is then  sufficient to
reduce to echelon form the matrix whose rows are generators of $\pi_2(\Ker \widetilde U) $: the indexes of the columns without pivots correspond to  the elements in $\calV_k$ to select for constructing $\calG_k$.
\end{proof}

The proof of the previous proposition guarantees the correctness of the following:

\underline{Algorithm MinimalBasis}

\underline {Input:} 
\begin{enumerate} 
\item [-] $s\in \N$
\item [-] $\{\calN_0, \ldots, \calN_{s}\}$  a complement of a homogeneous ideal $J$ up to degree $s$
\item[-] $\{\calV_1, \ldots, \calV_{s}\}$ where $\calV_k$ is a basis for $J \cap \langle \calN_{k-1}^+ \rangle $
\end{enumerate}

\underline{Output:} $ \{\calG_1,\ldots,\calG_s\}$ where:
\begin{enumerate}
\item [---] $\calG_k \subset \calV_k$ for each $k\in\{1, \ldots, s\}$
\item[---] the polynomials in $\calG_1 \cup
  \ldots \cup\calG_s$ are a minimal set of generators of the ideal $I(J_1,..,J_s)$
\end{enumerate}

\underline{Procedure:}
\begin{enumerate}
\item[-] $\calG_1 =\calV_1$
\item[-] for $k=2..s$ repeat
\begin{enumerate}
\item[] $S_1 :=(\partial \calN)_{k-1}^+ \setminus \calN_{k-1}^+$
\item[] $S_2:= (\partial\calN)_k$
\item[] $l:= | \calV_{k -1}^+ |$
\item[] $s_1 := | S_1 |$
\item[] $s_2 := | S_2 |$
\item[] --- construct the $ (s_1+s_2) \times (l+s_2)$  matrix with columns the $s_1+s_2$ 
\item[] --- coefficients w.r.t. $S_1 \cup S_2$ of the polynomials $v \in \calV_{k-1}^+ \cup \calV_k$.
\item[] $\widetilde U:=$ matrix (coeffs$(v,S_1\cup S_2) \ |\ v \in \calV _{k-1}^+ \cup \calV_k)$ 
\item[] --- compute  the intersection $\langle V_{k-1}^+ \rangle \cap  \langle \calV_k \rangle$
\item[]  $K := \Ker(\widetilde U)$
\item[]  $dk := | K |$
\item[] --- construct the $dk \times s_2$  matrix with rows the last $s_2$ entries of the
\item [] ---  vectors in $K$.
\item[] $MK:= matrix(\pi_2(v) \ | v \in K)$
\item[] $(RMK,\Sigma):= RRE(MK)$
\item[] ---  the polynomials in $\calV_k$ with index not in $\Sigma$ are put in $\calG_k$
\item[] $\calG_k:= \{ v_j \in \calV_k  \ |  j \not \in \Sigma\}$
\end{enumerate}
\end{enumerate}
\medskip

In the case when we start with a border basis, we can improve our previous construction. The computation of the generators of the intersection  $\langle V_{k-1}^+ \rangle \cap  \langle \calV_k \rangle$  can then be accomplished with only some column subtractions, and the construction of each level of the minimal basis  can be completed with one column  echelon reduction  of an $s_2 \times (l-s_1) $ matrix.  The new algorithm is based on the following Proposition whose proof is an immediate consequence of the properties of a border basis.

\begin{prop}\label{minimize-2}\label{rankpredict}  Let $s\in \N$ and let $J$ be a proper homogeneous ideal in
  $\mathcal P$. Assume that $\calN=\{\calN_0, \ldots, \calN_s\}$ is
a complement of  $J$ up to degree $s$ and let 
$\calB=\calB_1 \cup \ldots \cup \calB_s$ be the associated border
basis. With the notation of the previous proof, if $ \widetilde U= \left(\begin{array}{c|c}
U_1 & 0\\ \hline
U_2 & U_3
\end{array}\right) $ we have: 
\begin{enumerate}
\item[(i)] after reordering  the elements of $\calB_k$ we can assume that the $s_2 \times s_2$
  block $U_3$ is the identity matrix $I_{s_2}$
\item[(ii)] each element in $U_1$ is either $0$ or $1$; more precisely
  each row in $U_1$ contains at least one element equal to $1$
  and each column in $U_1$ contains at most one element equal to $1$
\item[(iii)]by means of finitely many
subtractions performed on the left $l$ columns of $\widetilde U$ we can reduce
$\widetilde U$ to the form
$$\widehat U=
\left(\begin{array}{c|c|c}
I_{s_1} & 0 & 0\\ \hline
P_1 & P_2 & I_{s_2}
\end{array}\right)$$
\item[(iv)] the set of the columns of the  $s_2 \times (l-s_1)$ matrix $P_2$ is a basis of $\pi_2(\Ker \widetilde U)=\pi_2(\Ker \widehat U)$
\item[(v)] $\widehat U$ has full row-rank and so 
$\dim \Ker \widetilde U = \dim \Ker \widehat U = l-s_1$
\item[(vi)] reducing $P_2$ to echelon  form by column operations, the pivots indicate the redundant  elements in  $\calB_k$.
\end{enumerate}
\end{prop}

\section{Curves given by points}\label{bypoints}

A natural application of the results of the previous section is the
construction of the ideal of an irreducible projective curve 
starting from the knowledge of a finite 
set of points on it. Let us recall the following classic result:

\begin{prop}\label{enough-points} 
Assume that $C$ is an irreducible projective curve in $\pp^n(\K)$ of
  degree $d$. Let $\calR=\{R_1, \ldots, R_h\}$ be a set of points on
  $C$.
\begin{enumerate}
\item For all $s\in \N$ such that $h>sd$, we have $\calI(C)_{\leq s} =
  \calI(\mathcal R)_{\leq s}$.
\item If $\calI(C)$ can be generated by polynomials of degree at most $m$
  and $h>md$,
  then $\calI(C)=\mathcal P\cdot \calI(\mathcal R)_{\leq m}$
\end{enumerate}
\end{prop}
\begin{proof} (1) It suffices to prove that 
$\calI(\mathcal R)_k \subseteq \calI(C)_k$ for all $k\leq s$. If $f\in
  \calI(\mathcal R)_k$, the polynomial $f$ vanishes on $h > sd \geq kd
  = \deg f \cdot \deg C$ points.  Since $C$ is irreducible, by
  B\'ezout's Theorem the hypersurface $V(f)$ contains $C$, i.e. $f\in
\calI(C)_k$.

(2) By hypothesis $\calI(C)=\mathcal P\cdot \calI(C)_{\leq m}$, thus the
    result follows immediately from (1).
\end{proof}

The previous result allows us to reduce the construction of
$\calI(C)$ to the computation of a set of generators for  the
ideal $J=\calI(\calR)$ where $\calR=\{R_1, \ldots, R_h\}$ is a set of
$h$ points in $\pp^n(\K)$. 
By Proposition \ref{Bkgenerates} this can be done by computing a
border basis of $J$. In the case of an ideal of points, we 
are able to compute $J_k \,\cap   \langle \calN_{k-1}^+ \rangle $ using the point
evaluation maps.

Assume that we have computed $\calN_{k-1}$ and let
  $\calN_{k-1}^+=\{m_1 ,\ldots, m_t\}\subseteq \calP_k$.  Consider the
  $h \times t$ evaluation matrix
$$M_{\mathcal R}=\left( \begin{array}{ccc} m_1(R_1) & \ldots &
m_t(R_1) \\ \vdots && \vdots \\ m_1(R_h) & \ldots &
m_t(R_h)\end{array} \right)$$ where, if $R_i=[r_{i,0}, \ldots,
r_{i,n}]$, by $m_j(R_i)$ we mean $m_j(r_{i,0}, \ldots, r_{i,n})$.
Note that the rank of the matrix $ M_{\mathcal R}$ and its null-space
$\Ker M_{\mathcal R}$ does not depend on the chosen representation of
the points in the projective space. Note also that each vector in
$\Ker M_{\mathcal R}$ is the vector of the coordinates of a polynomial
in $J_k \,\cap  \langle \calN_{k-1}^+ \rangle $ w.r.t. the basis $\{m_1 ,\ldots,
m_t\}$. In particular $\dim \Ker M_{\mathcal R} = \dim (J_k \,\cap
 \langle \calN_{k-1}^+ \rangle )$.

Performing  Gaussian elimination by rows,  followed if necessary by a permutation of the columns (which corresponds to a permutation of the basis $\{m_1 ,\ldots, m_t\}$), we can assume that $\calM_{\mathcal R} =
\left(\begin{array}{c|c} I_r & A\\ \hline 0 &
        0 \end{array}\right)$. In this way the columns of
    $\left(\begin{array}{c} - A\\ \hline I_{t-r}\end{array}\right)$ are a
    basis of the null-space $\Ker M_{\mathcal R}$. If we choose as
    $\calN_k$ the first $r$ monomials in the permuted basis, we have  the null-space in border form, which gives us $\calB_k$.  
\smallskip

We now want to estimate the complexity of our procedure to compute a
minimal set of generators up to degree $s$ of the ideal $J$ of $h$
distinct points. 

Our algorithm first computes a border basis up to degree $s$, then minimizes
this basis removing redundant elements. The basic tool is Gaussian
elimination; recall that the complexity of Gaussian
elimination performed on a $m \times p$ matrix  is
$O(mp \min(m,p))$.

As for the first phase to compute the border basis, the $k$-th step of
the recursive procedure described above to compute $\calN_k$ and
$\calB_k$ requires to perform Gaussian elimination on the $h \times t$
matrix $M_{\mathcal R}$. 
As already observed, $\dim \Ker M_{\mathcal R} = \dim (J_k \,\cap
 \langle \calN_{k-1}^+ \rangle )$; hence, by Remark \ref{compl-prop}, $\dim \Ker
M_{\mathcal R} = t- \dim  \langle \calN_k \rangle $. In particular $\dim
 \langle \calN_k \rangle =|\calN_k|=\rk M_{\mathcal R} \leq h$.  Thus for each $i=1,
\ldots, s$ we have that $|\calN_i|\leq h$ and hence $|\calN_i^+|\leq
(n+1)h$.  Therefore the complexity of each step of the algorithm is
$O(nh^3)$. As a consequence the complexity of the recursive algorithm
in $s$ steps to compute a border basis of $I(J_1, \ldots, J_s)$ is $O(snh^3)$.

As for the minimizing phase outlined in Proposition \ref{minimize-2},
note that for each $k$ we have that 
$|\calB_k|=|(\partial \calN)_k|\leq (n+1)h$, \ 
$|\calB_{k-1}^+|\leq (n+1)^2h$ and hence 
 $$|\widetilde{\calB_k}|=|\calB_{k-1}^+ \cup \calB_k| \leq
(n+1)(n+2)h.$$
Moreover the distinct monomials appearing in the polynomials of 
$\calB_{k-1}^+ \cup \calB_k$ form a subset of $\calN_{k-2}^{++}$ and
therefore they are at most $(n+1)^2h$.

Since the number of the left columns in $U$ is $|\calB_{k-1}^+|\leq
(n+1)^2h$, when we reduce the matrix $U$ to the form $\widehat U$ by
means of subtractions among the left columns, we need at most
$(n+1)^2h$ column subtractions. The length of each of these columns is
at most $(n+1)^2h$, thus the complexity of the reduction of $U$ to
$\widehat U$ is $O(n^4h^2)$. 

The last part of the minimizing phase requires to reduce the matrix
$P_2$ to echelon form by column operations. The number of rows of
$P_2$ is equal to $|(\partial \calN)_k| \leq (n+1)h$,
while the number of its columns is $ \leq (n+1)^2h$. Thus 
the complexity of the algorithm to reduce $P_2$ is $O(n^4h^3)$.
 
Hence the complexity of the algorithm outlined in Proposition \ref{minimize-2} to compute a minimal set of generators for $J_{\leq s}$ is $O(sn^4h^3)$.

\section{Curves in $\pp^n(\C)$ given by approximate points}\label{numeric}

In this section we consider the situation where the points
$\calR=\{R_1, \ldots, R_h\}$ on the irreducible projective curve $C$
in $\pp^n(\C)$ are given only approximately. In this case we can
perform the computations needed for the described procedure by
replacing Gaussian elimination by more suitable and numerically stable
tools.

As observed in Section \ref{BB}, if one is only interested in
computing a set of generators of $I(J_1, \ldots, J_s)$ (not
necessarily a border basis up to degree $s$), it is sufficient to
compute a basis of $J_k \,\cap
 \langle \calN_{k-1}^+ \rangle $  and compute a complement $\calN_k$ for each $k$ (see Remark \ref{complsuffices}).  The
former task corresponds to computing a basis of the null-space of the $h
\times t$ matrix
$$M_{\mathcal R}=\left( \begin{array}{ccc} m_1(R_1) & \ldots &
m_t(R_1) \\ \vdots && \vdots \\ m_1(R_h) & \ldots &
m_t(R_h)\end{array} \right)$$
where $\{m_1 .\ldots, m_t\}=\calN_{k-1}^+$.

A numerically stable way to compute both the rank and an
orthogonal basis of $\Ker M_{\mathcal R}$ is the SVD-algorithm which
assures that one can find a unitary $h \times h$ matrix $U$, a
unitary $t \times t$ matrix $V$ and an
$h \times t$ real matrix $\Sigma$ 
such that $M_{\mathcal R}=U\Sigma \overline V^t$; the elements
$\sigma_{ij}$ of the matrix $\Sigma$ are zero whenever $i\ne j$ and
for $i= 1, \ldots, l=\min\{h,t\}$ we have $\sigma_{1,1} \geq \ldots \geq \sigma_{l,l}\geq 0$. 

Either we know the rank of $M_{\mathcal R}$ (see for instance
Proposition \ref{dimIs}) or we can examine the
singular values $\sigma_i$ of $\Sigma$ in order to obtain a rank
determination as in \cite{Golub}. In any case if $\rk M_{\mathcal R}=r$,
then by the properties of the SVD decomposition the last $t-r$ columns
of $V$ are an orthogonal basis of $\Ker M_{\mathcal R}$.

In order to compute $\calN_k$  and continue to the next step, 
we take advantage of the stability properties of the QRP-algorithm which, given a matrix $M$, 
constructs a unitary matrix $Q$, a permutation matrix $P$ and an upper-triangular matrix $R=
\left(\begin{array}{c|c} R_1 & R_2 \end{array}\right)$ such that $S=QRP$. 
The permutation matrix $P$ exchanges columns in order to improve the condition number 
of the matrix $R_1$. If $M$ has full row-rank, then $R_1$ is invertible; otherwise it is possible to use the diagonal elements of $R_1$ to make a rank determination of $M$.

In our case we apply the QRP-algorithm to the  $(t-r)\times t$ matrix $S$  whose rows are the
last $t-r$ columns of $V$; recall that the columns of $S$ are indexed by $\calN_{k-1}^+$.
In the decomposition $S=QRP$ the columns of $R$ are a permutation of the columns 
of $S$ and the monomials corresponding to the columns of $R_1$ will be chosen 
to be border monomials, while the monomials corresponding to the columns of $R_2$ 
will be chosen as the complement $\calN_k$. 
Observe that the rows of $R$ correspond to a new basis for $J_k
\,\cap  \langle \calN_{k-1}^+ \rangle $.

If we want to compute a border basis of $I(J_1, \ldots, J_s)$, we can compute the
matrix $R_1^{-1}R=\left(\begin{array}{c|c} I & R_1^{-1}R_2
\end{array}\right)$ whose rows correspond to a basis $\calB_k$ of
$\Ker M_{\mathcal R}$ consisting of border polynomials. 

Otherwise, if we want to compute a minimal set of
generators of $I(J_1, \ldots, J_s)$, we can proceed as described in the algorithm
  MinimalBasis using SVD to compute kernels and QRP
  to select stable pivot columns.  Using the notation of Proposition \ref{minimize-1}, 
the first step is to compute a basis of $\Ker \widetilde U$. 
In order to do this, we use an SVD construction taking into account that, by Proposition \ref{rankpredict} (v)
 $ \dim \Ker \widetilde U = \dim \Ker  U = l -s_1$.
We then apply the QRP-algorithm to the matrix $N$ whose rows contain the projection 
by $\pi_2$ of a set of generators of $\Ker \widetilde U$ and whose columns are indexed by the generators of 
$J_k\,\cap  \langle \calN_{k-1}^+ \rangle $.  We thus obtain $N=Q'R'P'$. Examining the diagonal elements of 
$R'$ we can determine its rank $r'$; the columns of $R'$ correspond to a permuted basis 
of $J_k\,\cap  \langle \calN_{k-1}^+ \rangle $ and the generators corresponding to the first $r'$ 
columns are redundant and can be discarded.

\smallskip

Exact Gaussian elimination, SVD and QRP-algorithm applied to an $m
\times n$ matrix all have the same complexity $O(mn \min(m,n))$.
In the approximate algorithm the computation of the null-space using
SVD has a complexity $O(nh^3)$
and is followed by a QRP-algorithm which has a complexity $O(n^3h^3)$.
Thus the complexity to compute a set of generators
or a border basis of $I(J_1, \ldots, J_s)$ is $O(sn^3h^3)$, while the
complexity to compute a minimal set of generators of $I(J_1, \ldots,
J_s)$ by the algorithm MinimalBasis    is $O(sn^6h^3)$.

Alternatively, after computing a border basis, we can give a numerical algorithm based on Proposition   \ref{minimize-2}, which would reduce the complexity to $O(sn^4h^3)$ with a slight loss of numerical precision.

\begin{ex} We implemented our algorithm in Octave. Here is the result we obtained when we tested it   on the following parametric sextic space curve $C$ taken from \cite{Jia}:

$\begin{array}{cl}
x =\! & 3 s^4 t^2-9 s^3 t^3-3 s^2 t^4+12 st^5+6 t^6 \\
y =\! & -3 s^6 + 18 s^5 t -27 s^4 t^2-12 s^3 t^3+33 s^2 t^4+6 s t^5-6 t^6 \\
z =\! & s^6 - 6 s^5 t + 13 s^4 t^2 -16 s^3 t^3 + 9 s^2 t^4 + 14 s t^5 -6 t^6 \\
w =\! & -2 s^4 t^2 + 8 s^3 t^3-14 s^2 t^4+20 s t^5-6 t^6.
\end{array}
$

By Theorem \ref{Peskine} the ideal of this curve of degree 6 in $\pp^3(\C)$ can be generated by polynomials of degree at most $5$.

We chose $31 > 6\cdot 5$ points using roots of unity of the following form:
$$  s=1;\quad   t = exp(2\pi i k/31) \quad k\in\{1, \ldots, 31\}.$$

Running the algorithm BorderBasisWithComplement, we obtained no polynomials of degree $1$ (showing that the ideal is not contained in any hyperplane), no polynomials of degree $2$, $4$ polynomials of degree $3$, $11$ of degree $4$ and $22$ of degree $5$, yielding a set of generators for the ideal in $.08$ seconds. 
Among the $20$ monomials of degree $3$ the QRP-algorithm chose  $ {z \sp 2} x, {y x w}, {y \sp 2}  w, {z \sp 2} w$ as border monomials and the remaining $16$ as generators of a complement.

We obtained a minimal basis consisting of only the 4 polynomials of degree 3 in an additional time of $.02$ seconds: 
\smallskip

$\begin{array}{ll}
f_1=\! & {\bf {z \sp 2} x} + {0.0666666666 6} \,{z \sp 2} y+ {{0.0680555555 5} \, z 
{y \sp 2}}  - {0.0361111111} \, z y x  \\
 & -{{0.2833333333} \, z  y w} -{{0.55} \, z  {x \sp 2}} -{{1.066666667} \,  z  x w}+ {{0.01527777778}\, {y \sp 3}}   \\
 & - {{0.0916666666 6} \,  {y \sp 2} x}+{{0.3055555555} \, y {x \sp 2}}+{{0.1833333334} \, 
{x \sp 2} w},
\end{array}$
\smallskip

$\begin{array}{ll}
f_2 =\! &  {\bf {y x w}}  + {0.2} \,  {z \sp 2} y + {{0.1416666667} \,  z  {y \sp 2}} - {{0.4833333333} \,  z yx} 
-{{0.1} \,  z  y  w} \\
 & -{{0.9} \,  z {x \sp 2}}   - {{0.2000000001} \,  z x w} + {{0.025} \,  {y \sp 3}} -{{0.15} \,  {y \sp 2}  x} + {{0.5} \,  y {x \sp 2}}  \\
 & + {{0.3000000001} \,  {x \sp 2}  w},
\end{array}$
\smallskip

$\begin{array}{ll}
f_3 =\! & {\bf {{y \sp 2}  w}} -{{0.8} \,  {z \sp 2} y} - {{0.3166666667} \,  z {y \sp 2}} - 
{{0.5666666667} \,  z yx} + {{0.4} \,  z  y  w}  \\
 & + {{0.6000000001} \,  z  {x \sp 2}} + {{0.8000000002} 
\, z  x  w} -{{0.0166666666 6} \,  {y \sp 3}} \\
 & + {{0.0999999999 9} \,  {y \sp 2} x}  -{{0.3333333333} \,  y  {x \sp 2}} 
- {{0.2000000002} \,  {x \sp 2}  w},
\end{array}$
\smallskip

$\begin{array}{ll}
f_4 =\! & {\bf {{z \sp 2} w}}-{{0.6666666667} \,  {z \sp 3}} 
-{{0.162962963} \,  {z \sp 2} y}+{{0.0604938271 7} \,  z  
{y \sp 2}}  \\ 
 & - {{0.0320987654 5} \  z y  x}+{{0.9703703704} \,  z y  w} 
-{{0.4888888888} \,  z {x \sp 2}} \\
& +{{0.3851851853} \,  z x  w}  -{{0.1666666667} \,  z  {w \sp 2}}+{{0.0135802469 1} \,  {y \sp 3}} \\
 & -{{0.0814814814 9} \,  {y \sp 2} x}+{{0.2716049383} \,  y  {x \sp 2}} 
-{{0.9444444445} \,  y  {w \sp 2}}  \\
& + {{0.1629629629} \,  {x \sp 2}  w} 
-{{0.2222222224} \,  x  {w \sp 2}}
\end{array}$
\smallskip

Using continued fractions, we then attempted to convert the coefficients from floating point numbers  to rational numbers, obtaining the following polynomials:
\smallskip

$\begin{array}{ll}
f_1 =\! & {\bf {{z \sp 2} x}} +{{1 \over {15}} \  {z \sp 2} y}+{{{49} \over {720}} \  z  {y \sp 
2}}+{{{11} \over {720}} \  {y \sp 3}}-{{{13} \over {360}} \  
z y  x} -{{{11} \over {120}} \  {y \sp 2} x} -{{{11} \over {20}} \  z 
{x \sp 2}} \\
 & +  {{{11} \over {36}} \  y  {x \sp 2}} -{{{17} \over {60}} \  z 
 y  w} -{{{16} \over {15}} \  z  x w}+{{{11} \over {60}} \  {x \sp 
2} w}, 
\end{array}$
\smallskip

$\begin{array}{ll}
f_2 =\!& {\bf {y x  w}}+{{1 \over 5} \  {z \sp 2} y}+{{{17} \over {120}} \  z {y 
\sp 2}}+{{1 \over {40}} \  {y \sp 3}} -{{{29} \over {60}} \  z y  x} 
-{{3 \over {20}} \  {y \sp 2}  x} -{{9 \over {10}} \  z  {x \sp 2}}+{{1 
\over 2} \  y  {x \sp 2}}  \\
 & - {{1 \over {10}} \  z  y  w} -{{1 \over 5} \  
z  x  w}+{{3 \over {10}} \  {x \sp 2}  w}, 
\end{array}$
\smallskip

$\begin{array}{ll}
f_3 =\! & {\bf {{y \sp 2}  w}}-{{4 
\over 5} \  {z \sp 2}  y} -{{{19} \over {60}} \  z  {y \sp 2}} -{{1 \over 
{60}} \  {y \sp 3}} -{{{17} \over {30}} \  z y  x}+{{1 \over {10}} \  {y 
\sp 2} x}+{{3 \over 5} \  z {x \sp 2}} -{{1 \over 3} \  y  {x \sp 
2}} \\
 & + {{2 \over 5} \  z y w}+{{4 \over 5} \  z x 
w} -{{1 \over 5} \  {x \sp 2}  w}, 
\end{array}$
\smallskip

$\begin{array}{ll}
f_4 =\! & {\bf {{z \sp 2} w}}-{{2 \over 3} \  {z \sp 3}} -{{{22} 
\over {135}} \  {z \sp 2} y}+{{{49} \over {810}} \  z {y \sp 2}}+{{{11} 
\over {810}} \  {y \sp 3}} -{{{13} \over {405}} \  z y x} -{{{11} \over 
{135}} \  {y \sp 2} x}  \\
 & -{{{22} \over {45}} \  z  {x \sp 2}} + {{{22} \over 
{81}} \  y {x \sp 2}}+{{{131} \over {135}} \  z y w}+{{{52} \over {135}} \  z x  w}+{{{22} \over {135}} \  {x \sp 2}  w} 
-{{1 \over 6} \,  z  {w \sp 2}} -{{{17} \over {18}} \  y  {w \sp 2}}  \\
 & -{{2 \over 9} \,  x {w \sp 2}}.
 \end{array}$
 
The floating point coefficients were sufficiently accurate to recover the exact rational coefficients and the previous polynomials generate the exact ideal of the curve over $\Q$.
\end{ex}

\section{Degree bounds for ideal generators}\label{bounds}

Let $C$ be an irreducible projective curve (seen as a set of points in
$\pp^n(\K)$) and let $I=\mathcal I(C)\subset \mathcal P=\K[x_0,
\ldots, x_n]$ be the prime homogeneous ideal consisting of all the 
polynomials vanishing on $C$.  In Section \ref{bypoints}
we saw that the computation of $I$ can be reduced to the computation
of the ideal of sufficiently many points on the curve (see Proposition
\ref{enough-points}). In order to be sure that one has enough points,
it is necessary to  bound the degrees of
generators of the ideal $I$. Such a bound can be obtained from the
\emph{regularity} of the curve.

Recall that if $0 \to \ldots \to F_1 \to F_0 \to J \to 0$ is a graded
 free resolution of a polynomial ideal $J$ , we say that $J$ is
$k$-regular if each $F_j$ can be generated by polynomials of degree
$\leq k+j$. Then we call \emph{regularity} of $J$ the integer $reg
(J)= \min \{k\ |\ J \mbox{ is $k$-regular}\}$.

\smallskip

If $reg(C)=reg(\mathcal I(C))=m$, then (see for instance \cite{Eisenbud-syz})
$I$ can be generated by homogeneous polynomials of degree at most $m$;
moreover the Hilbert function $HF_I(s)$ and the Hilbert polynomial
$HP_I(s)$ of $\mathcal P /I$ take the same value for all integers $s\geq
m$.

The next result gives information about the regularity of the curve as
a function of the curve degree and of the dimension of the embedding
projective space:
 
\begin{teo}[Gruson-Lazarsfeld-Peskine \cite{Peskine}]\label{Peskine}
  Let $\mathcal D\subset \pp^n(\K)$ be a reduced and irreducible curve
  of degree $d$ not contained in any hyperplane, with $\K$
  algebraically closed and $n\geq 3$. Then $reg(\mathcal D)\leq
  d-n+2.$ \\ Moreover, if $\mathcal D$ has genus $g>1$ then
  $reg(\mathcal D)\leq d-n+1.$
\end{teo}

The previous result and many of the degree bounds apply to curves not
contained in any hyperplane, i.e. non-degenerate, which happens if and
only if $I_1= \mathcal I(C)_1=(0)$. On the other hand a degenerate
curve in $\pp^n(\K)$ is a non-degenerate curve in the projective
subspace $V(I_1)$ having dimension $n-\dim I_1$. So the bounds of this
section apply to degenerate curves if we replace $n$ by $n-\dim I_1$.

\begin{rem}\label{paramcurve} A parametrization of a curve
  $C\subset \pp^n(\K)$ can be regarded as a machine delivering as many
  points on the curve as needed. Moreover if $\psi\colon \pp^1(\K) \to
  \pp^n(\K)$ is a rational map whose image is the curve $C$  and
  $\psi([t_0,t_1])= [F_0(t_0,t_1), \ldots, F_n(t_0,t_1)]$ with $F_i(t_0,t_1)$
  homogeneous polynomials of degree $s$, then $\deg C \leq s$. 
Using this upper bound for the curve degree, via
  Theorem \ref{Peskine} we obtain an upper bound for the degrees of
  generators of $C$. This allows us to use our procedure to
  compute generators of the ideal $I$ of the curve and in particular
  to compute an implicit representation of $C$ starting from a
  parametric one.  \qed\end{rem}

Sharper bounds for the degrees of generators can be obtained
for certain curves obtained as the image of the embedding given by a
complete linear series. A first result in this direction concerns
canonical curves (see for instance \cite{Saint-Donat-Petri}):

\begin{teo}[Petri \cite{Petri}] \label{Petri} The ideal of a non-hyperelliptic
  canonical curve of genus $g\geq 4$ can be generated by polynomials
  of degree 2 and of degree 3.
\end{teo}

\begin{prop}\label{appl-Petri} Let $C\subset  \pp^n(\K)$ be a
  non-degenerate irreducible curve of geometric genus $g\geq 4$ and of degree
 $d=2g-2$, with $\K$ algebraically closed and $n=g-1$. Then the ideal
 $I=\mathcal I (C)$ can be generated by 
 polynomials of degree 2 and of degree 3.
\end{prop}
\begin{proof} By Theorem VI.6.10 in \cite{Walker} any non-degenerate
 curve of genus $g$ and degree $2g-2$  embedded in $\pp^{g-1}(\K)$ is
 a smooth non-hyperelliptic canonical curve. Then the conclusion follows
 from  Theorem \ref{Petri}.
\end{proof}

The following result of Saint-Donat gives bounds for the
degrees of generators of curves of genus $g$ which are the image of an
embedding given by a complete linear series of sufficiently high
degree:

\begin{teo}\label{Saint-Donat-CR}(\cite{Saint-Donat-CR}) Let $C\subset
  \pp^n(\K)$ be an irreducible nonsingular curve of genus $g$
 embedded by a complete linear series
  of degree $d$. 
\begin{enumerate}
\item If $d \geq 2g+1$, then the ideal $I$ of $C$ can be generated
  by polynomials of degree 2 and of degree 3.
\item If $d \geq 2g+2$, then
  $I$ can be generated by polynomials of degree 2.
\end{enumerate}
\end{teo}

The following result shows that we can remove the condition of being
embedded by a complete linear series:

\begin{prop}\label{appl-Saint-Donat-CR} Let $C\subset  \pp^n(\K)$ be
  a non-degenerate irreducible curve of geometric genus $g$ and degree $d$, with $\K$
 algebraically closed and $n=d-g$. 
Then
\begin{enumerate}
\item If $d \geq 2g+1$, then the ideal $I=\mathcal I (C)$  can be generated
  by polynomials of degree 2 and of degree 3.
\item If $d \geq 2g+2$, then
  $I=\mathcal I (C)$ can be generated by polynomials of degree 2.
\end{enumerate}
\end{prop}
\begin{proof} Recall that any curve $C$ of degree $d$ and genus $g$ in
  $\pp^n(\K)$ can be seen as the image of the map given by a linear
series contained in $\mathcal L(D)$ for some divisor $D$ of degree $d$
on a non-singular model $\widetilde C$ of $C$.

For any divisor $D$ on $\widetilde C$, denote by $l(D)$ the dimension
of the complete linear series $\mathcal L(D)$. We also denote by $K$ a
canonical divisor.

If $d \geq 2g+1$, then $l(K - D)=0$ and hence by Riemann-Roch Theorem
$l(D)=d-g+1$; moreover (see \cite{Fulton}) the map induced by $D$ is
an embedding. Thus we can see $C$ as the image of an embedding of
$\widetilde C$ in $\pp^{d-g}(\K)$.  So, since $C$ is non-degenerate,
$C$ is embedded by a complete linear series if and only if $n=d-g$.

Moreover, if $d \geq 2g+1$ and $n=d-g$, the curve $C$ is
the image of a non-singular curve through a map which is an embedding,
therefore $C$ is non-singular. Hence the hypotheses of Theorem
\ref{Saint-Donat-CR} are fulfilled and it implies our result. 
\end{proof}

{\bf Example 1.} Assume that $C$ has genus $g\geq 3$ and is embedded
by the bicanonical map. Since $\deg(2K)=4g-4\geq 2g+2$, by Theorem
\ref{Saint-Donat-CR} $I$ can be generated by quadratic
polynomials. 

{\bf Example 2.} If $C$ has genus $g=2$ and is embedded by the
tricanonical map (which gives an embedding because $\deg(3K)=6g-6 =6
\geq 2g+1$), since $\deg(3K)\geq 2g+2$ again by Theorem
\ref{Saint-Donat-CR} $I$ can be generated by polynomials of degree 2.
\smallskip

Other results giving bounds for the degrees of generators of the curve
ideal in different situations are available in the literature; the
following one (see \cite{Akahori}) gives results in the hyperelliptic case:

\begin{prop}[Akahori]\label{Akahori}  Let $C\subset  \pp^n(\K)$ be
  an irreducible non-degenerate and non-singular hyperelliptic curve
  of genus $g$ and degree $d$.
\begin{enumerate}
\item If  $d=2g$, then the ideal $I$ of $C$ can be generated
  by polynomials of degrees 2, 3 and 4.
\item If  $d=2g-1$, then  $I$ can be generated
  by polynomials of degrees 2, 3, 4  and 5.
\end{enumerate}
\end{prop}

When we need to compute the null-space of the evaluation matrix 
$M_{\mathcal R}$ for approximate points (see Section \ref{numeric}),
we can try to determine its rank by inspecting its singular value
spectrum. Since this is not guaranteed to work, it is useful to
predict what the rank should be. Since the rank of $M_{\mathcal R}$
equals $|\calN_k|= \dim \mathcal P_k - \dim I_k$, the following
proposition gives a situation where we can predict this value.

Recall that if  $g_0, \ldots, g_n$ is a basis of
the complete linear series
$\mathcal L(D)$ of dimension $l(D)=n+1$, then  $I_k= \Ker \varphi_k$
where $\varphi_k \colon \mathcal P_k \to \mathcal L(kD)$ defined by 
 $\varphi_k (x_i)=g_i$.

\begin{prop}\label{dimIs} If $\varphi_k$ is surjective and 
$k\cdot \deg(D) \geq 2g-1$,
  then $$ \dim \mathcal P_k -\dim I_k = k\cdot\deg(D)-g+1.$$
\end{prop}
\begin{proof} If $\varphi_k$ is surjective, then $\dim I_k = \dim \Ker
  \varphi_k=\dim   \mathcal P_k - \dim\mathcal L(kD)$. \\
Since $\deg(kD)=k\cdot\deg(D)\geq 2g-1$, then $i(kD)=0$. Hence by 
 Riemann-Roch Theorem we get $l(kD)= \deg(kD) -g+1=k \cdot\deg(D)
 -g+1$ and the conclusion follows. 
\end{proof}

{\bf Examples.} 1. Assume that $C$ is non-hyperelliptic of genus
$g\geq 4$ and take the canonical divisor $D=K$; in particular
$\deg(D)=2g-2$ and $l(D)=g$ (i.e. in this case $n=g-1$). By Theorem
\ref{Petri} we already know that $I$ can be generated by polynomials
of degree 2 and of degree 3, and so it suffices to know $\dim I_2$ and
$\dim I_3$.
 
Then by Noether's Theorem the map $\varphi_k \colon \mathcal P_k
\to \mathcal L(kD)$ is surjective for all $k$. Furthermore for all
$k\geq 2$ we have that $k\cdot\deg(D)\geq 2g-1$, hence by Proposition
\ref{dimIs} we have 
$\dim \mathcal P_k -\dim I_k =k(2g-2)-g+1.$
In particular 
$$|\calN_2|=\dim \mathcal P_2 - \dim I_2= 2(2g-2)-g+1=3g-3$$
$$|\calN_3|=\dim \mathcal P_3- \dim I_3=3(2g-2)-g+1= 5g-5.$$

\smallskip

2. If we choose a divisor $D$ such that $\deg(D)\geq 2g+1$, then
   $\varphi_k \colon \mathcal P_k \to \mathcal L(kD)$ is surjective
   (see \cite{Mumford}); moreover for all $k\geq 1$ we have that
   $k\cdot \deg(D) \geq 2g-1$. Then we can compute $\dim I_k$ by
   Proposition \ref{dimIs}.

In particular if $D=2K$ and $g\geq 3$, then we know that  $I$ can
be generated by quadratic polynomials. Since $i(2K)=0$, by
Riemann-Roch Theorem $l(2K)=3g-3$, i.e. $n=3g-4$. Then 
$$|\calN_2|=\dim \mathcal P_2 - \dim I_2=2(4g-4)-g+1=7g-7.$$

If $g=2$ and we choose $D=3K$, we know that $I$ can be generated by
quadratic polynomials. Since $i(3K)=0$, by Riemann-Roch Theorem
$l(3K)=5g-5=5$, i.e. $n=4$.  Then 
$$|\calN_2|=\dim \mathcal P_2 - \dim I_2=2 \deg(3K)-g+1=11.$$

\begin{rem} When the ideal can be generated in degree $2$ and we know 
$|\calN_2|$ as above, then the algorithm to compute a minimal set of generators can be considerably simplified. If $M_{\mathcal R}$ is the point evaluation matrix for all monomials of degree $2$, since we know that $\rk M_{\mathcal R}=|\calN_2|$, a single application of the SVD-algorithm with this imposed rank computes a basis for the null-space of $M_{\mathcal R}$ which directly yields a minimal set of generators of the ideal.
\end{rem}

\section*{Acknowledgments}
We wish to thank Mika Sepp\"al\"a for proposing this problem to us and a referee for many
helpful comments.

\bibliographystyle{amsalpha}

\end{document}